\documentclass[11pt]{article}

\usepackage[a4paper,margin=1.15in]{geometry}
\usepackage{enumitem}
\usepackage{amsmath, amsfonts, amsthm, amssymb}
\usepackage{mathrsfs}
\usepackage{booktabs}
\usepackage{array}
\usepackage{bbm}
\usepackage{apptools}
\usepackage{chngcntr}
\usepackage{xfrac}
\usepackage[hidelinks]{hyperref}
\AtAppendix{\counterwithin{theorem}{section}}
\newcommand {\ti}{\rightarrow \infty}

\newcommand {\e}{\epsilon}

\newcommand {\p}{\mathbb{P}}
\newcommand {\x}{\mathbb{E}}

\newcommand {\ind}{\mathbbm{1}}

\newtheorem{theorem}{Theorem}
\newtheorem*{theorem*}{Theorem}
\newtheorem{lemma}[theorem]{Lemma}
\newtheorem{proposition}[theorem]{Proposition}

\newtheorem{definition}{Definition}[section]
\newtheorem{remark}{Remark}
\newtheorem{hypothesis}{Hypotheses}

\let\emptyset\varnothing
\title{Critical branching processes in random environment with immigration and an application to randomised reproducing graphs}
\author{
Simon Irons\\
\small University of Sheffield
\and
Jonathan Jordan\\
\small University of Sheffield
}
\date{}

\begin{document}

\maketitle

\begin{abstract}
We study branching processes in an i.i.d.\ random environment with immigration in the critical regime, where the underlying offspring mechanism satisfies the critical condition that the log of the average population growth, across environments, and before immigration, is zero.   In this setting environmental fluctuations are balanced on average, and the long-term behaviour is determined by the interaction between these fluctuations and the immigration sequence.   While recurrence and transience criteria for critical BPREI were established by Bauernschubert (2014), the possibility of null recurrence remained unresolved. 

We show that, under natural integrability assumptions on the offspring and immigration distributions, a critical BPREI is null recurrent.  In particular, the process returns to zero infinitely often but admits no stationary distribution.  Our results close a gap in the classification of the critical regime and provide a structural understanding of the balance between environmental variability and immigration.

As an application, we resolve the open critical case of the Randomised Reproducing Graph (`RRG') model introduced by Jordan (2011), showing that in the critical regime the proportion of vertices of a fixed degree admits no limiting distribution.
\end{abstract}
\section{Introduction and main result}
Branching processes in a random environment (`BPRE') and branching processes with immigration (`BPI') have each been well studied; see, for example, \cite{AGKV05,AK71,SW69}.  However, material combining both features, branching processes in a random environment with immigration (BPREI), is comparatively sparse, and even in the relatively tractable setting of i.i.d.\ environments and immigration, the literature is often not straightforward to navigate.  Examples of BPREI literature include \cite{KE87,BS14,RO07}.

In this paper we consider the critical i.i.d.\ case, where the underlying offspring mechanism satisfies
\begin{equation*}\x_\mathcal E[\log \mu_i] = 0,\end{equation*}
that is the log of the average population growth $\mu_i$, across environments $\mathcal E$, before immigration, is zero. In this regime the process is finely balanced with environmental fluctuations competing with immigration.  Progress in understanding this balance was made by Bauernschubert \cite{BS14}, who established criteria distinguishing recurrence from transience in the critical BPREI, identifying a threshold in terms of the tail behaviour of the immigration distribution.  However, the possibility of null recurrence was not addressed.  In the present paper we resolve this question and prove, in Theorem~\ref{th:null}, that under a natural set of assumptions the critical BPREI is null recurrent.

This work was motivated by a gap in the classification of the Randomised Reproducing Graph (RRG) introduced by Jordan \cite{JO11}. 
In that model the asymptotic behaviour falls into a recurrent/transient dichotomy depending on whether certain growth parameters are subcritical or supercritical.  The critical case, however, remained open and has proved delicate, resisting approaches based on Foster--Lyapunov techniques and related methods.  Our main result enables us to close this gap: in Theorem~\ref{th:RRG} we show that, in the critical RRG, the proportion of vertices with a given degree admits no limiting distribution.

The paper is organized into a section on definitions and theory and a section that uses the RRG problem as an example of application.

\section{Definitions and theory}
To prove null recurrence of the critical irreducible BPREI, we will argue by contradiction.  We suppose the process is positive recurrent so that it admits a unique stationary distribution $\pi$.  In particular, if $(Z_n)_{n\ge0}$ is started with $Z_0\sim\pi$, the process is stationary and its law is invariant under time shifts.

By shift-invariance we may extend the process to a version $(Z_n)_{n\in\mathbb Z}$ indexed by the full integer line such that $(Z_n)_{n\ge -m}$ has the same law as the original stationary process for every $m\ge0$.  Thus the stationary law may be realized at all times $-n$, and we may analyse the population at time $0$ as the accumulated contribution of immigrants arriving in the remote past.  We begin by defining the random environment which is the underlying probability space on which randomness is defined.  
\begin{definition}[The random environment]
Let $(\mathcal E_n)_{n\in\mathbb Z}$ be an i.i.d.\ sequence of random variables, called the \emph{environment}.  For each $n\in\mathbb Z$,
the value of $\mathcal E_n$ determines the reproduction and immigration laws in force between generations $n$ and $n+1$.  Conditional on $(\mathcal E_n)_{n\in\mathbb Z}$, we define:
\begin{itemize}
	\item For each $n\in\mathbb Z$, a sequence $(\xi_{n,i})_{i\ge1}$ of offspring random variables that are identically distributed over $i$.
	\item For each $n\in\mathbb Z$, an immigration random variable $M_n$.
\end{itemize}

We assume the following independence, given $(\mathcal E_n)_{n\in\mathbb Z}$
\begin{itemize}
	\item For each fixed $n$, the variables $(\xi_{n,i})_{i\ge1}$ are independent for all $n$ and $i$ and are independent of $M_n$.
	\item For $n\neq m$, the collections $\{(\xi_{n,i})_{i\ge1},\,M_n\}$ and $\{(\xi_{m,i})_{i\ge1},\,M_m\}$ are independent.
\end{itemize}
\end{definition}
We write
\begin{equation*}\mu_n := \mathbb E[\xi_{n,1}\mid \mathcal E_n],\qquad\sigma_n^2 := \operatorname{Var}(\xi_{n,1}\mid \mathcal E_n)\end{equation*}
for the conditional mean and variance of the offspring distribution at time $n$.  Next we define one-sided branching processes that are those that start at a time $n\in \mathbb Z$. 
\begin{definition}[Branching process started at time $n$]
A \emph{branching process in the random environment} $(\mathcal E_k)_{k\in\mathbb Z}$ \emph{with immigration} started at time $n\in\mathbb Z$ is a process $(Z^{(n)}_{n+k})_{k\ge0}$, with  $Z^{(n)}_n\in\mathbb Z_{\ge0}$ and defined recursively by
\begin{equation}
	Z^{(n)}_{n+k+1}=\sum_{i=1}^{Z^{(n)}_{n+k}} \xi_{n+k,i}+M_{n+k+1},\qquad k\ge0.\label{eq:BP1}
\end{equation}
A \emph{branching process in the random environment} $(\mathcal E_k)_{k\in\mathbb Z}$ is defined using \eqref{eq:BP1}  but with $M_{n+k+1}=0$ a.s. for all $k\ge 0$.
\end{definition}
In section 2 where we refer to both BPREI and BPRE, to distinguish when we specifically mean a BPRE rather than a BPREI, we will use $\hat{Z}_n$ for the BPRE.

A two sided version of a BPREI is given by $\lim_{n\ti} Z^{(-n)}_{-n+k+1}$, that is stochastic processes that satisfy the following definitions.
\begin{definition}[Two-sided branching process in a random environment]
A \emph{two-sided branching process in the random environment} $(\mathcal E_k)_{k\in\mathbb Z}$ \emph{with immigration} is a process $(Z_n)_{n\in\mathbb Z}$ with $Z_n\in\mathbb Z_{\ge0}$ for all $n\in\mathbb Z$ satisfying
\begin{equation}
	Z_{n+1}=\sum_{i=1}^{Z_n} \xi_{n,i}+M_{n+1},\qquad n\in\mathbb Z,\label{eq:BP2}
\end{equation}
whenever such a process exists.
\end{definition}
\begin{remark}
A two-sided process $(Z_n)_{n\in\mathbb Z}$ satisfying \eqref{eq:BP2} does not exist in general.  However, if the one-sided process
$(Z^{(0)}_k)_{k\ge 0}$ is positive recurrent and so admits a stationary probability distribution $\pi$ on $\mathbb N_0$, then there exists a
stationary version $(Z_n)_{n\in\mathbb Z}$ with $Z_0\sim\pi$ satisfying \eqref{eq:BP2} for all $n\in\mathbb Z$.
\end{remark}
We make the following additional assumptions
\begin{hypothesis}[H]\label{hy:standing}
Let $f_n$ be the probability generating function for $\xi_{n,i}=\xi_{n,1}$ so that $\mu_n=f'_n(1)$.
\begin{enumerate}
	\item[\textbf{(H1)}] \textbf{Criticality and non-degeneracy}\begin{equation*}\x[\log \mu_0]=0,\qquad\p[\mu_0=1]<1.\end{equation*}
	\item[\textbf{(H2)}] \textbf{Reproduction probability and moment assumptions}  $\p[\xi_{0,1}=0]>0$ and let $f$ denote the generic offspring p.g.f.\ so that $f\stackrel{d}{=}f_0$, then there exists, $C$ and $\delta>0$, such that
	\begin{equation*}f''(1)<\infty\;\text{a.s.},\qquad\frac{f''(1)}{(f'(1))^2}<C\;\text{a.s.},\qquad\x\big[|\log f'(1)|^{2+\delta}\big]<\infty.\end{equation*}
	\item[\textbf{(H3)}] \textbf{Immigration assumptions.} There exists $\delta>0$ such that
	\begin{equation*}\p[M_0\ge1]>0,\qquad\p[M_0=0]>0,\quad \x[(\log^+ M_0)^{2+\delta}]<\infty.\end{equation*}
\end{enumerate}
\end{hypothesis}

Let $X_i:=\log \mu_{i-1}$, then, as the environment variables are i.i.d., $(X_i)_{i\ge1}$ is an i.i.d.\ sequence and $\x[X_i]=\x[X_1]$.  Depending on whether $\x[X_1]$ is less than, equal to or greater than 0 the BPREI is said to be sub-critical, critical or super-critical.  Here we are interested in the critical case $\x[X_i]=\x[X_1]=0$.

Geiger and Kersting consider critical branching processes in a random environment for $(\hat{Z}_n)_{n\geq 0}$ in \cite{GK01} for a one-sided process.  In the two sided setting we are interested in the probability of offspring of an individual at time $-n$ surviving until time 0.  The event that a single individual present at time $-n$ has descendants alive at time 0 depends only on the reproduction and environments at times $-n+1, -n+2,\dots, 0$.  Because the environments are i.i.d.\ this block has the same law as a block of the first $n$ generations in the one-sided model.  By shift invariance, the probability of survival from time $-n$ to time 0 is therefore the same as the probability that the one-sided process survives for $n$ generations.  We define this probability as $q_n$, that is,
\begin{equation}
	q_n=\p[\hat{Z}_n>0\mid \mathcal E_0,\dots, \mathcal E_{n-1}]=\p[\hat{Z}_0^{(-n)}>0\mid \mathcal E_{-n+1},\dots, \mathcal E_{0}]\label{eq:qn1}.
\end{equation}
We define the realized probability generating function of the offspring random variables $\xi_{k,i}\overset{\mathrm{d}}{=}\xi_{k,1}$ to be $f_k$, that is, for $0\le s\le 1$
\begin{equation*}
	f_k(s):=\x[s^{\xi_{k,1}}\mid \mathcal E_k].
\end{equation*}
So that the conditioning in \eqref{eq:qn1} can be equivalently written
\begin{equation*}
	q_n=\p[\hat{Z}_n>0\mid f_0,\dots, f_{n-1}]=\p[\hat{Z}_0^{(-n)}>0\mid f_{-n+1},\dots, f_{0}].
\end{equation*}
Let $X_i:=\log f'_{i-1}(1)$ and define a random walk associated with the BPRE by $S_0^{(n)}=0$,
\begin{equation*}
	S_k^{(n)}:=\sum_{i=1}^k X_{-n+i}\qquad k=1,\dots,n
\end{equation*}
and
\begin{equation*}
	W_n:=\sum_{k=0}^{n}\exp(-S_k^{(n)}).
\end{equation*}
In Appendix A we formalise and extend a discussion in \cite{GK01} and use this in the following proposition.
\begin{proposition}[Survival of lineages from $-n$ infinitely often]\label{pr:qnio}
Under hypotheses (H1) and (H2) there exists $\e>0$ such that
\begin{equation*}
	\p[\;q_n>\e\; \text{i.o.}\;]=1.
\end{equation*}
\end{proposition}
\begin{proof}
We note that $q_n$, $S_k^{(n)}$ and $W_n$ are re-indexed versions of $q_n$, $S_k$ and $W_n$ in Appendix A with the index running from $-n$ rather than $0$.  Hence, by Proposition \ref {pr:surv} and the assumption \emph{(H2)} of a uniform bound across an environment (which are i.i.d.\ across $k\in \mathbb Z_+$) we have
\begin{equation*}
	q_n^{-1}\le C W_n.
\end{equation*}
Now $S_k^{(n)}$ is a random walk counting forwards from $-n$ to 0 so we can redefine the walk, without altering $W_n$, counting backwards from 0 to $-n$ so that $S_0^{(n)}:=0$ and for $k\ge 1$
\begin{equation*}
	S_k^{(n)}=\sum_{j=0}^{k-1}X_{-(n-1+j)}\qquad k=0,\dots,n.
\end{equation*}
Then
\begin{equation*}
	S_k^{(n)}=X_{-(n-1)}+\sum_{j=1}^{k-1}X_{-(n-1+j)}=X_{-(n-1)}+S_{k-1}^{(n-1)}
\end{equation*}
and
\begin{align*}
	W_n&=1+\sum_{k=1}^{n}\exp(-S_k^{(n)})=1+\sum_{k=1}^{n}\exp(-X_{-(n-1)}-S_{k-1}^{(n-1)})\\
	&=1+\exp(-X_{-(n-1)})W_{n-1}.
\end{align*}
This is the setting of Babillot, Bougerol and Ellie in \cite{BBE97} in one dimension with $A_n=\exp(-X_{-(n-1)})$ and $B_n=1$.  We check the hypotheses (H) of this paper
\begin{enumerate}
	\item The affine map $x\mapsto A_1 x + B_1$ is not almost surely the identity.  Since $B_1=1$, this holds trivially.
	\item For some $\delta>0$ we require $\x[(|\log(A_1)|+\log^+||B_1||)^{2+\delta}]<\infty$ so that $\x[|X_0|^{2+\delta}]<\infty$ which is the case by assumption.
	\item $\x[\log A_1]=\x[-X_0]=0$ and $A_1\not\equiv 1$.  The first follows from criticality, and the second from $\p[X_0=0]<1$.
\end{enumerate}
Hence, by Corollary 4.2 of \cite{BBE97}, $W_n$ visits any open set of positive $m$-measure i.o.\ where $m$ is an unbounded, invariant Radon measure, that is if $U$ is open with $m(U)>0$ then $\p[W_n\in U\;\text{i.o.}]=1$.

We next show that there exists a bounded $U$ with positive Radon measure.  Since $\x[\mu_0]=1$ and $\p[\mu_0=1]<1$, it follows that $\p[\mu_0>1]>0$ and $\p[\mu_0<1]>0$. Hence, $\p[\sfrac{1}{\mu_0}<1]>0$.  Recall that $W_n$ satisfies the affine recursion
\begin{equation*}
	W_n=1+A_nW_{n-1},\qquad\text{where}\qquad A_n=\exp(-X_{-(n-1)})=\frac{1}{\mu_{-(n-1)}}.
\end{equation*}
Since $\mu_{-(n-1)}$ is identical in distribution to $\mu_0$ it follows $\mathbb P[A_n<1]>0$.  Choose $a<1$ such that $\mathbb P[A_n\le a]>0$.  If for $k$ consecutive steps we have $A_{t+1},\dots,A_{t+k}\le a$, then iterating the recursion gives
\begin{equation*}
	W_{t+k}\le 1+a+\dots+a^k+a^kW_t = \frac{1-a^{k+1}}{1-a}+a^kW_t.
\end{equation*}
Since $a<1$, the right-hand side converges to $\frac{1}{1-a}$ as $k\to\infty$.  Let $M=\frac{1}{1-a}+1$ then, for each $W_t$ there exists $K_t\in\mathbb N$ such that
\begin{equation*}
	\frac{1-a^{K_t+1}}{1-a}+a^{K_t}W_t<M.
\end{equation*}
Hence, on the event $\{A_{t+1}\le a,\dots,A_{t+K_t}\le a\}$, we have $W_{t+K_t}\in(0,M)$.  Since the $A_n$ are i.i.d.\ and $\p[A_1\le a]>0$,
\begin{equation*}
	\p[A_{t+1}\le a,\dots,A_{t+K_t}\le a]=\p[A_1\le a]^{K_t}>0.
\end{equation*}
Hence, starting from any state, there is positive probability of entering $(0,M)$ after finitely many steps, so $(0,M)$ is accessible.  Therefore $(0,M)$ is a nonempty open accessible set, and so the invariant Radon measure $m$ given by Corollary 4.2 of \cite{BBE97} satisfies $m((0,M))>0$.  Let $\e=\sfrac{1}{CM}$. Then, recalling that $q_n^{-1}\le C W_n$
\begin{equation*}
	\p[W_n\in(0,M)\;\text{i.o.}]=1 \implies \p[q_n^{-1}<CM\;\text{i.o.}]=1 \implies \p[q_n>\e\;\text{i.o.}]=1.
\end{equation*}
\end{proof}

We next bridge this proposition in a lemma that gives the survival of infinitely many immigrant lineages in a BPREI and follow this with a lemma to show irreducibility and recurrence of the BPREI.
\begin{lemma}[Immigrants descend infinitely often]\label{lem:Bn_io}
Suppose the one-sided BPREI $(Z^{(0)}_k)_{k\ge 0}$ admits a stationary distribution $\pi$. Then it admits a stationary two-sided version
$(Z_n)_{n\in\mathbb Z}$.  For $n\ge1$ let $B_n$ be the event in the two-sided BPREI that there exists an immigrant at time $-n$ with at least one descendant alive at time $0$. Then under hypotheses (H)
\begin{equation*}
	\p[B_n\ \text{i.o.}]=1.
\end{equation*}
\end{lemma}
\begin{proof}
Let $E$ be the set of environments where $q_n>\e$ for infinitely many $n$, that is $E=\{\mathcal E:\ q_n(\mathcal E)>\e \text{ for infinitely many }n\}$ then $\p[E]=1$.  Let $\mathcal E\in E$ and
\begin{equation*}
	\mathcal N=\{n\ge1 : q_n>\e\},
\end{equation*}
which is infinite for $\mathcal E$.  For each $n\in\mathcal N$, on the event $\{M_{-n}\ge1\}$ select one immigrant at time $-n$
and follow only its descendants. Conditional on the environment, this lineage evolves as a one-ancestor BPRE and survives to time $0$ with probability $q_n$. Let $p=\p[M_0\ge1]$, then as $p>0$ and the immigration random variables are independent,
\begin{equation*}
	\p[B_n \mid \mathcal E] \ge p\,q_n \ge p\e \qquad \text{for all } n\in\mathcal N.
\end{equation*}
Hence,
\begin{equation*}
	\sum_{n\in\mathcal N} \p[B_n\mid\mathcal E]=\infty.
\end{equation*}
Conditional on $\mathcal E$, the immigration variables $(M_k)$ are independent across times and, within each generation, the offspring variables $(\xi_{k,i})_{i\ge1}$ are independent across $i$.  Since for each $n$ we follow only one selected immigrant at time $-n$, distinct immigrant lineages use disjoint collections of offspring variables. Hence the events $B_n$ are independent given $\mathcal E$, and by the conditional Borel--Cantelli lemma,
\begin{equation*}
	\p[B_n\ \text{i.o.}\mid\mathcal E]=1.
\end{equation*}
Taking expectations
\begin{align*}
	\p[B_n\ \text{i.o.}]&=\x[\ind_{\{B_n\;\text{i.o.}\}}]=\x[\x[\ind_{\{B_n\;\text{i.o.}\}}\mid \mathcal E]]\\
	&=\x[\p[B_n\ \text{i.o.}\mid \mathcal E]]=\x[1]=1.
\end{align*}
\end{proof}
\begin{lemma}[Irreducibility and recurrence of the BPREI]\label{lem:irred-rec}
Assume hypotheses (H), then the one-sided BPREI population process $(Z_k^{(0)})_{k\ge0}$ is a recurrent, irreducible Markov chain.
\end{lemma}
\begin{proof}
\emph{Irreducibility}
By (H2), $\mathbb P[\xi_{0,1}=0]>0$, so from any state $k\ge0$ there is positive probability that all $k$ individuals produce no
offspring in the next generation. By (H3), $\mathbb P[M_1=0]>0$.  Hence from any state $k$ there is positive probability to move to state $0$ in one step.  Conversely, by (H3), $\mathbb P[M_1\ge1]>0$, so from state $0$ there is positive probability to move to a positive state.  Since reproduction and immigration both have non-degenerate distributions, every state in $\mathbb N_0$ can be reached from 0 in finitely many steps with positive probability and the chain is irreducible.

\emph{Recurrence}
Under (H) the environment is critical with finite variance of the associated random walk increments and a $(2+\delta)$–moment condition on immigration. These are exactly the hypotheses of Theorem 3 in Bauernschubert \cite{BS14}, which states that the critical BPREI is recurrent.
\end{proof}
We are now in a position to prove our main result.
\begin{theorem}[Null recurrence]\label{th:null}
The critical BPREI process $(Z^{(0)}_{k})_{k\ge 0}$ under hypotheses (H) is null recurrent.
\end{theorem}
\begin{proof}
We know that the process is recurrent and irreducible from Lemma \ref{lem:irred-rec} so it remains to show that it is null.

Assume that $(Z_k^{(0)})_{k\ge 0}$ is positive recurrent.  Then, as it is irreducible it admits a unique stationary probability distribution $\pi$ on $\mathbb N_0$. Let $(Z_k)_{k\in \mathbb Z}$ be the two sided version of the BPREI.  Since $\pi$ is stationary, we have
\begin{equation*}
	Z_k \sim \pi\qquad \text{ for all }k\in \mathbb Z.
\end{equation*}
Hence, for every $\e>0$ there exists $M\in\mathbb N$ such that
\begin{equation}
	\p[Z_0 \ge M]=\pi(\{M,M+1,\dots\}) < \e.\label{eq:null1}
\end{equation}
Let $B_m$ be the event that there exists an immigrant at time $-m$ which has at least one descendant alive at time $0$.
For each such $m$, define
\begin{equation*}
	J_m := \#\{\text{immigrants at time }-m \text{ whose lineage survives to time }0\}.
\end{equation*}
then $B_m=\{J_m\ge 1\}$. As distinct immigration times give disjoint families at time $0$ we have
\begin{equation}
	Z_0 \;\ge\; \sum_{m=1}^n J_m \;\ge\; \sum_{m=1}^n \ind_{B_m}.\label{eq:null2}
\end{equation}
By Lemma \ref{lem:Bn_io}, the events $B_m$ occur infinitely often almost surely and
\begin{equation*}
	\sum_{m=1}^\infty \ind_{B_m} =\infty \;\text{a.s.}
\end{equation*}
In particular, for any fixed $M$,
\begin{equation*}
	\ind_{\{\sum_{m=1}^n \ind_{B_m} \ge M\}} \uparrow 1 \qquad \text{a.s. as } n\to\infty.
\end{equation*}
By monotone convergence and \eqref{eq:null2}
\begin{equation*}
	\p[Z_0 \ge M] \geq\p\bigg[\sum_{m=1}^n \ind_{B_m} \ge M\bigg] \rightarrow 1\qquad \text{ as }n\ti,
\end{equation*}
which contradicts \eqref{eq:null1}.  Hence $(Z_k^{(0)})_{k\ge 0}$ is not positive recurrent and, being recurrent and irreducible, must be null recurrent.
\end{proof}

\section{Null recurrence of the critical RRG}
Jordan introduced the randomised reproducing graph model (`RRG') in \cite{JO11}.  This iterative model for the sequence of graphs $(G_{n})_{n\geq 0}$ depends on three parameters $\alpha$, $\beta$, $\gamma$, starts with a simple graph $G_{0}$, and forms a new graph $G_{n+1}$ from $G_{n}$ by adding a child vertex for every vertex of $G_{n}$ and edges according to some stochastic rules governed by the parameters.  Significantly all edges in $G_{n}$ are retained into $G_{n+1}$.  In the model in this paper we retain the rules for vertex and edge creation but allow for edge deletion through a fourth parameter $\delta$.

To describe the model we refer (with a slight abuse of notation) to parent vertices in $G_{n}$ and their continuation into $G_{n+1}$ as $u$ and $v$ and child vertices as $u'$ and $v'$ respectively.  The edges of $G_{n+1}$ are obtained according to the following mechanism.  For each $n$ define independent (of each other and of the random variables at other stages of the construction) Bernoulli random variables $a^{(n)}_{\{ u,v \} }\sim Ber(\alpha)$ for each unordered pair $\{ u,v\}$ of vertices of $G_{n}$, $b_{u}^{(n)}\sim Ber(\beta)$ for each vertex in $G_{n}$, $c^{(n)}_{(u,v)}\sim Ber(\gamma)$ for each ordered pair $(u,v)$ of vertices of $G_{n}$ and $d^{(n)}_{\{ u,v \} }\sim Ber(\delta)$ for each unordered pair $\{ u,v\}$ of vertices of $G_{n}$.  We then connect and delete vertices as follows:
\begin{itemize}
    \item $u'$ is connected to $v'$ in $G_{n+1}$ if and only if  $a^{(n)}_{\{ u,v \} }=1$ and $u$ and $v$ are connected in $G_{n}$, so each child is connected to each of its parent's  neighbours' children with probability $\alpha$;
    \item $u$ is connected to $u'$ in $G_{n+1}$ if and only if $b_{u}^{(n)}=1$, so each child is connected to its parent with probability $\beta$;
    \item $u$ is connected to $v'$ in $G_{n+1}$ if and only if  $c^{(n)}_{(u,v)}=1$ and $u$ and $v$ are connected in $G_{n}$, so each child is connected to each of its parent's neighbours with probability $\gamma$; and
    \item having added the $\alpha$, $\beta$ and $\gamma$ edges we delete the edge $\{u,v\}$ precisely when $d^{(n)}_{\{ u,v \} }=0$.
\end{itemize}
It follows that the extended model is equivalent to the original model when $\delta=1$.

We apply Theorem \ref{th:null} to the example of the randomised reproducing graph.  Let $D_{n}$ be the degree of a vertex chosen at random from $V(G_{n})$.  Then we can define $D_{n}$ as a discrete time Markov chain iteratively.  If $\beta<1$ then the state space is $\mathbb{Z}_{\geq 0}$ and if $\beta=1$ the state space is $\mathbb{N}$.  When iterating from $G_n$ to $G_{n+1}$, write $v_{0n}\in V(G_{n+1})$ for the parent, or continuation, vertex $v_n\in V(G_n)$ and $v_{1n}\in V(G_{n+1})$ for the child vertex of $v_n\in V(G_n)$.  Let $v_{\emptyset}$ be chosen uniformly at random from $G_{0}$ and for $n\geq 1$ let
\begin{align*}
	v_{n+1}=\begin{cases}
		v_{0n}& \text{with probability }1/2\text{, a continuation vertex}\\
		v_{1n}& \text{with probability }1/2\text{, a child vertex}
	\end{cases}
\end{align*}
then,
\begin{align}
	D_{n+1}=(1-\zeta_{n+1})W_{n+1}+X_{n+1}+Y_{n+1}+\zeta_{n+1} Z_{n+1},\label{eq:Dn}.
\end{align}
Where $W_{n+1}\sim Bin(D_{n},\alpha)$, $X_{n+1}\sim Ber(\beta)$, $Y_{n+1}\sim Bin(D_{n},\gamma)$, $Z_{n+1}\sim Bin(D_{n},\delta)$ and $\zeta_{n+1}\sim Ber(1/2)$.

It follows that, provided $\beta>0$, we can consider $(D_n)_{n\geq 0}$ as a BPREI with the environment representing whether the chain follows a continuation or child vertex, each with probability $1/2$.  The immigration is represented by $\beta$-edges and the $X_n$ Bernoulli random variable.  If the environment is continuation then the offspring are determined by a sum of $\delta$ and $\gamma$-edges that arise from the Bernoulli random variables summed in the $Z_n$ and $Y_n$ random variables.  If the environment is child then the offspring are determined by a sum of $\alpha$ and $\gamma$-edges that arise from the Bernoulli random variables summed in the $W_n$ and $Y_n$ random variables.  We are interested in the critical case for the RRG which is $(\alpha+\gamma)(\delta+\gamma)=1$.
\begin{center}
\begin{tabular}{lll}
\toprule
\textbf{Parameter regime} & \textbf{Type of process $(D_n)$} & \textbf{Behaviour} \\
\midrule
$\beta=0$, $\alpha\neq\delta$ 
& Critical BPRE 
& $D_n \to 0$ a.s. \\

$\beta=0$, $\alpha=\delta$ 
& Critical Galton--Watson process 
& $D_n \to 0$ a.s. \\

$\beta>0$, $\alpha\neq\delta$ 
& Critical BPREI 
& Null recurrent (developed below) \\

$\beta>0$, $\alpha=\delta$ 
& Critical BPI 
& Null recurrent (developed below) \\
\bottomrule
\end{tabular}
\end{center}
The environments, immigration random variables and offspring random variables are all independent and identically distributed.  The offspring p.g.f.\ is given by
\begin{equation*}f_n(s)=\begin{cases}
		(1-\gamma)(1-\delta)+s(\gamma(1-\delta)+\delta(1-\gamma))+s^2\delta\gamma& \text{with probability }\frac{1}{2}\\
		(1-\gamma)(1-\alpha)+s(\gamma(1-\alpha)+\alpha(1-\gamma))+s^2\alpha\gamma& \text{with probability }\frac{1}{2}
	\end{cases}\end{equation*}
so that
\begin{equation*}\mu_1=f'(1)=\begin{cases}
		\delta+\gamma& \text{with probability }\frac{1}{2}\\
		\alpha+\gamma& \text{with probability }\frac{1}{2}.
	\end{cases}\end{equation*}
Therefore, $\x[\log\mu_1]=\log[((\alpha+\gamma)(\delta+\gamma))^{\frac{1}{2}}]=0$.  Also $\p[\mu_1=1]=1$ requires $\delta+\gamma=\alpha+\gamma=1 \iff \delta=\alpha$ which is not the case so $\p[\mu_1=1]<1$.  All moments of the offspring random variables are finite and
\begin{equation*}\frac{f''(1)}{(f'(1))^2}=\begin{cases}
		\sfrac{2\delta\gamma}{(\delta+\gamma)^2}& \text{with probability }\frac{1}{2}\\
		\sfrac{2\alpha\gamma}{(\alpha+\gamma)^2}& \text{with probability }\frac{1}{2}
	\end{cases}\;\leq 1.\end{equation*}
Finally, the immigration assumptions hold so that all hypotheses (H) hold and we can conclude the following theorem.
\begin{theorem}[Critical RRG has no limiting degree distribution] \label{th:RRG}
In the critical randomised reproducing graph with $\beta>0$ and $\alpha\neq\delta$, the degree of a randomly chosen vertex $D_n$ is null recurrent, $\p[D_n=d \;\text{i.o.}\;]=1$ for each fixed $d$ and $\limsup_{n\ti} D_n=\infty$. Moreover, the degree proportions $p_d^{(n)}$ in $G_n$ admit no limiting distribution.
\end{theorem}
\begin{proof} By Theorem~\ref{th:null}, $D_n$ is a null recurrent Markov chain and therefore has no invariant probability distribution and as $D_n$ is irreducible the rest of the statements about it follow. Hence the laws $\mathcal L(D_n)$ do not converge. If $p_d^{(n)}$ converged in probability to some $p_d$, then $\p[D_n=d]=\x[p_d^{(n)}]\rightarrow p_d$, giving a limiting distribution for $D_n$ which is a contradiction. Therefore $p_d^{(n)}$ has no limiting distribution.
\end{proof}
In the case $\beta>0$ and $\alpha=\delta$, the environment is constant so that the process reduces to a BPI and we use some classic results in \cite{AN04}.

\begin{proposition}\label{pr:crit0}
In the critical randomised reproducing graph with $\beta>0$ and $\alpha=\delta$, $\p[D_n=d]\rightarrow 0$ and $p^{(n)}_d\rightarrow 0$ almost surely as $n\ti$ for all $d\in \mathbb{Z}_{\geq 0}$.
\end{proposition}
\begin{proof}
When $\alpha=\delta$ the environment is deterministic and the process $D_n$ reduces to a critical branching process with immigration.  This is an application of Theorem 1, Section VI.7 of \cite{AN04} noting that $f'(1)=1$ and $f''(1)<\infty$ where $f$ denotes the offspring p.g.f.\ of the associated BPI.  Applying this to $D_n$ gives $\lim_{n\ti}\p[D_{n}=d]=0$ and the result for $p^{(n)}_d$ follows from the argument used in Theorem \ref{th:RRG}.
\end{proof}

For $\alpha=\delta$ we can go further and prove that in the critical case a scaled degree of a vertex chosen at random converges to a specific Gamma distribution and do so in the following proposition.
\begin{proposition}\label{pr:critGamma}
In the critical randomised reproducing graph with $\beta>0$ and $\alpha=\delta$,
\begin{equation*}
	\frac{D_n}{\gamma(1-\gamma)n}\;\Rightarrow\;D
\end{equation*}
where $D\sim\Gamma\!\left(\frac{\beta}{\gamma(1-\gamma)},1\right)$.
Consequently,
\begin{equation*}
	p_d^{(n)}\to 0
\end{equation*}
for each fixed $d$.
\end{proposition}
\begin{proof}
The offspring process has mean $f'(1)=1$, and variance $\sigma^2=f''(1)=2\gamma(1-\gamma)$. The immigrant process has constant mean $\beta$. The result for $D_n$ then follows directly from an application of Theorem 4 of section VI.7 of \cite{AN04} and the result for $p^{(n)}_d$  follows from the argument used in Theorem \ref{th:RRG}.
\end{proof}
\begin{remark}
In the case $\alpha=\delta$ and criticality leads to $\gamma=1-\alpha$, and the process $D_n$ reduces to a critical branching process with immigration. The offspring distribution has mean $1$ and variance $\gamma(1-\gamma)$, so each immigrant lineage contributes a random amount of order one and fluctuations governed by this variance.  The population at time $n$ can be written as a sum of contributions from immigrants arriving at times $1,\dots,n$,
\begin{equation*}
	D_n \approx \sum_{k=1}^n Y_k,
\end{equation*}
where the $Y_k$ are approximately independent contributions from each lineage. Thus $D_n$ is the accumulation of $n$ random contributions, each with fluctuation scale $\gamma(1-\gamma)$, leading to the scaling $\gamma(1-\gamma)n$.  In this setting the limiting law is, somewhat surprisingly, explicitly Gamma. The shape parameter $\beta/\gamma(1-\gamma)$ represents the ratio between the immigration rate and the variance scale of a single lineage.

When $\beta=\gamma(1-\gamma)$, this ratio is $1$, so the Gamma limit reduces to an exponential distribution. This corresponds to the variability of individual lineages being averaged out by immigration.
\end{remark}
\appendix
\section{Results derived from Geiger \& Kersting}
The topic of Geiger \& Kersting is BPRE.  The offspring of an individual in $n$-th generation has p.g.f.\ $f_n$.  Some further p.g.f.\ notation is introduced in \cite{GK01}.  Let $f_{n,n}(s):=s$, $a_0:=1$ and 
\begin{align*}
	f_{k,n}&:=f_k\circ f_{k+1}\circ \dots\circ f_{n-1},&& 0\le k\le n-1,\\
	a_k&:=(f'_0(1)f'_1(1)\dots f'_{k-1}(1))^{-1}, && k\ge 1,\\
	g_k(s)&:=\frac{1}{1-f_k(s)}-\frac{1}{f'_k(1)(1-s)}, && 0\le s\le 1,\\
	\eta_{k,n}&:=g_k(f_{k+1,n}(0)).
\end{align*}
Here, $X_i:=\log \mu_{i-1}=\log f'_{i-1}(1)$, $S_n:=X_1+\dots +X_n$ is the random walk associated with the BPRE and $W_n:=\sum_{k=0}^{n}\exp(-S_k)$.  We formalise and extend a discussion in \cite{GK01} to arrive at the following proposition.
\begin{proposition}\label{pr:surv}
The survival probability $q_n$ is given by
\begin{equation}
	q_n=\left(\exp(-S_n)+\sum_{k=0}^{n-1} \eta_{k,n}\exp(-S_k)\right)^{-1}.\label{eq:qn3}
\end{equation}
Furthermore, if $f''_k(1)<\infty$ and $\sfrac{f''_k(1)}{(f'_k(1))^2}\le C$ where $C\ge 1$ for all environments then
\begin{equation}
	q_n\geq C^{-1}W_{n}^{-1}.\label{eq:qn4}
\end{equation}
\end{proposition}
\begin{proof}
For $0\le s\le 1$
\begin{align*}
	\x[s^{Z_n}\mid Z_0,\dots,Z_{n-1},f_0,\dots,f_{n-1}]
	&=\x\left[ \prod_{i=1}^{Z_{n-1}} s^{\xi_{n-1,i}} \middle|  Z_0,\dots,Z_{n-1},f_0,\dots,f_{n-1} \right]\\
	&=\prod_{i=1}^{Z_{n-1}} \mathbb E\!\left(s^{\xi_{n-1,i}} \mid f_{n-1}\right)\\
	&=\big(f_{n-1}(s)\big)^{Z_{n-1}}.
\end{align*}
Then, as $Z_{n-1}$ is measurable with respect to $\sigma(f_0,\dots,f_{n-1})$ and independent of $f_{n-1}$
\begin{align*}
	\x[s^{Z_n}|f_0,\dots,f_{n-1}]&=\x\big[\x[s^{Z_n}|Z_0,\dots,Z_{n-1},f_0,\dots,f_{n-1}]|f_0,\dots,f_{n-1}\big]\\
	&=\x\big[\bigl(f_{n-1}(s)\bigr)^{Z_{n-1}}|f_0,\dots,f_{n-1}\big]\\
	&=\x\big[\bigl(f_{n-1}(s)\bigr)^{Z_{n-1}}|f_0,\dots,f_{n-2}\big].
\end{align*}
Now the r.h.s.\ is just the l.h.s.\ but with $s\mapsto f_{n-1}(s)$, $Z_n\mapsto Z_{n-1}$ and $f_0,\dots,f_{n-1}\mapsto f_0,\dots,f_{n-2}$ so that an iteration gives
\begin{equation*}
	\x[s^{Z_n}|f_0,\dots,f_{n-1}]=f_0(f_1(\cdots f_{n-1}(s)\cdots)).
\end{equation*}
Hence,
\begin{equation*}
	\p[Z_n=0\mid f_0,\dots,f_{n-1}]=\x[0^{Z_n}|f_0,\dots,f_{n-1}]=f_0(f_1(\cdots f_{n-1}(0)\cdots))
\end{equation*}
and
\begin{equation*}
	q_n=1-f_0(f_1(\cdots f_{n-1}(0)\cdots)).
\end{equation*}
The following identity is shown in \cite{GK01} for $0\le s\le 1$
\begin{align*}
	\frac{1}{1-f_0(f_1(\cdots f_{n-1}(s)\cdots))}=\frac{a_n}{1-s}+\sum_{k=0}^{n-1}a_k g_k(f_{k+1,n}(s)).
\end{align*}
Then setting $a_k=\exp(-S_k)$ in the identity and evaluating at 0 gives Equation \eqref{eq:qn3}.  In Lemma 2.1 of \cite{GK01} the bound
\begin{equation*}
	g_k(s)=\frac{1}{1-f_k(s)}-\frac{1}{f'(1)(1-s)}\le \frac{f''(1)}{(f'(1))^2}
\end{equation*}
is established, provided $f''(1)<\infty$.  Applying this to Equation \eqref{eq:qn3}, substituting $f_{k+1,n}(s)$ for $s$ and evaluating at $0$ gives Inequality \eqref{eq:qn4}.
\end{proof}

\end{document}